\documentclass[12pt]{article}

\usepackage{amsmath,enumerate,amsfonts,amssymb,color,graphicx,amsthm,stmaryrd}

\usepackage[normalem]{ulem}

\usepackage{hyperref}

\def\RR{{\mathbb R}}
\def\HH{{\mathbb H}}

\def\SSphere{{\mathbb S}}

\newcounter{marnote}

\begin{document}
\newtheorem{THM}{Theorem}
\renewcommand*{\theTHM}{\Alph{THM}}
\newtheorem{Def}{Definition}[section]
\newtheorem{thm}[Def]{Theorem}
\newtheorem{lem}[Def]{Lemma}
\newtheorem{question}[Def]{Question}
\newtheorem{prop}[Def]{Proposition}
\newtheorem{cor}[Def]{Corollary}
\newtheorem{clm}[Def]{Claim}
\newtheorem{step}[Def]{Step}
\newtheorem{sbsn}[Def]{Subsection}
\newtheorem*{conj}{Conjecture}
\theoremstyle{remark}
\newtheorem{rem}[Def]{Remark}

\numberwithin{equation}{section}

%%%%%%%%%%%%%%%%%%%%%%%%%%%%%%%%%%%%%%%%%%%%%%%%%%%%%%%%%%%%%%%%%%%%%%%%%%%%%%%%%%%%%

\title
{Existence and uniqueness to a fully non-linear version of the Loewner-Nirenberg problem}

\author{Maria del Mar Gonz\'alez \footnote{Universidad Aut\'onoma de Madrid,
Departamento de Matem\'aticas, Campus de Cantoblanco, 28049 Madrid, Spain. Email: mariamar.gonzalezn@uam.es}~ , YanYan Li \footnote{Department of Mathematics, Rutgers University, Hill Center, Busch Campus, 110 Frelinghuysen Road, Piscataway, NJ 08854, USA. Email: yyli@math.rutgers.edu.}~ and Luc Nguyen \footnote{Mathematical Institute and St Edmund Hall, University of Oxford, Andrew Wiles Building, Radcliffe Observatory Quarter, Woodstock Road, Oxford OX2 6GG, UK. Email: luc.nguyen@maths.ox.ac.uk.}}

\date{}
\maketitle

\begin{abstract}
We consider the problem of finding on a given Euclidean domain $\Omega$ of dimension $n \geq 3$ a complete conformally flat metric whose Schouten curvature $A$ satisfies some equation of the form $f(\lambda(-A)) = 1$. This generalizes a problem considered by Loewner and Nirenberg for the scalar curvature. We prove the existence and uniqueness of such metric when the boundary $\partial\Omega$ is a smooth bounded hypersurface (of codimension one). When $\partial\Omega$ contains a compact smooth submanifold $\Sigma$ of higher codimension with $\partial\Omega\setminus\Sigma$ being compact, we also give a `sharp' condition for the divergence to infinity of the conformal factor near $\Sigma$ in terms of the codimension.
\end{abstract}

\tableofcontents

\section{Introduction}

Assume $n \geq 3$. For a positive $C^2$ function $u$ defined on an open subset of $\RR^n$, define its conformal Hessian to be
\begin{equation}\label{conformal-Hessian}
A^u = -\frac{2}{n-2} u^{-\frac{n+2}{n-2}}\,\nabla^2 u + \frac{2n}{(n-2)^2} u^{-\frac{2n}{n-2}} \nabla u \otimes \nabla u - \frac{2}{(n-2)^2}\,u^{-\frac{2n}{n-2}}\,|\nabla u|^2\,I.
\end{equation}
In differential geometry, $A^u$ is  the Schouten curvature tensor of the metric  $u^{\frac{4}{n-2}}g_{\rm flat}$ where $g_{\rm flat}$ is the Euclidean metric. We will write $\lambda(-A^u)$ to denote the eigenvalues of $-A^u$.

Let
\begin{equation}
\Gamma\subset \RR^n\ \mbox{be an open symmetric cone with vertex at the origin}
\label{01}
\end{equation}
satisfying
\begin{equation}
\Gamma \supset \Gamma + \Gamma_n = \{\lambda + \mu: \lambda \in \Gamma, \mu \in \Gamma_n\}
\label{02}
\end{equation}
where $
\Gamma_n 	:=\{\mu\in \RR^n\ |\
\mu_i>0\ \forall\ i\}$.
Assume that
\begin{equation}
f\in C^0(\overline \Gamma)\ \mbox{is
symmetric in
} \ \lambda_i,
\label{03}
\end{equation}
\begin{equation}
f>0 \text{ in } \Gamma,
f=0\ \mbox{on}\ \partial \Gamma,
\text{ and } f(\lambda + \mu) \geq f(\lambda) \text{ for all }\lambda \in \Gamma, \mu \in \Gamma_n,
\label{04}
\end{equation}
\begin{equation}
f \text{ is homogeneous of some positive degree}.
	\label{f-homo}
\end{equation}
To keep the notation simple, we use the convention that whenever we write $f(\lambda)$, we assume $\lambda \in \bar\Gamma$. Note that we do not require that $f$ nor $\Gamma$ be convex.

Let $\Omega \subset \RR^n$ be a (bounded or unbounded) domain in $\RR^n$, i.e. $\Omega$ is an open and connected subset of $\RR^n$. Consider the equation
\begin{equation}
f(\lambda(-A^u)) = 1, \qquad u  > 0 \qquad  \text{ in } \Omega
	\label{Eq:X1}
\end{equation}
subject to the boundary condition
\begin{equation}
u(x) \rightarrow +\infty  \text{ as } \textrm{dist}(x,\partial\Omega) \rightarrow 0.
	\label{Eq:X1BC}
\end{equation}

A similar problem can be posed when $\Omega$ is a subset of a given Riemannian manifold, but this will not be pursued in the present paper.

The typical example is when $f$ is built from the $\ell$-th elementary symmetric function of the eigenvalues $\lambda=(\lambda_1,\ldots,\lambda_n)$, i.e.,
$$\sigma_\ell(\lambda)=\sum_{i_1<\ldots<i_\ell} \lambda_{i_1}\ldots\lambda_{i_\ell}, \quad\text{for} \quad\ell=1,\ldots,n.$$
The corresponding cone $\Gamma$ is the cone $\Gamma_\ell=\{\lambda=(\lambda_1,\ldots,\lambda_n):\sigma_1(\lambda)>0,\ldots,\sigma_\ell(\lambda)>0\}$. In this setting, \eqref{Eq:X1} is a fully non-linear (non-uniformly) elliptic equation of Hessian type, which is usually referred to as the $\sigma_\ell$-Yamabe problem in the `negative case'. Although its counterpart in the `positive case', i.e. the equation $f(\lambda(A^u)) = 1$, has been intensively studied after the works of Viaclovsky \cite{Viac00-Duke, Viac00-AMS} and Chang, Gursky and Yang \cite{CGY02-AnnM,CGY03-IHES,CGY03-IP} (see e.g. Ge and Wang \cite{GeWang06}, Guan and Wang \cite{GW03-IMRN}, Gursky and Viaclovsky \cite{GV07}, Li and Li \cite{LiLi03,LiLi05}, Li \cite{Li09-CPAM},  Li and Nguyen \cite{LiNgPoorMan}, Sheng, Trudinger and Wang \cite{STW07}, Trudinger and Wang \cite{TW09} and the references therein), the equation \eqref{Eq:X1} has received much less attention. We recall, for instance, Gurksy and Viaclovsky \cite{Gursky-Viaclovsky:negative-curvature}  where the negative $\sigma_k$ problem for a modified Schouten tensor was solved (compare Li and Sheng \cite{Li-Sheng:flow} where a flow method is used instead),  Guan \cite{Guan:negative-Ricci}, Gursky, Streets and Warren \cite{Gursky-Streets-Warren} and Sui \cite{Sui} for the analogue for the Ricci tensor. Note that the equations considered in \cite{Guan:negative-Ricci, Gursky-Streets-Warren, Gursky-Viaclovsky:negative-curvature, Li-Sheng:flow, Sui} can be recast in the form \eqref{Eq:X1} for a suitable $f$.\\

In the particular case that $\ell=1$, $\sigma_1$ is the scalar curvature times a positive constant, and  \eqref{Eq:X1} reduces to the Yamabe equation
\begin{equation}\label{Eq:Yamabe}
\Delta u=u^{\frac{n+2}{n-2}}, \qquad u  > 0 \qquad  \text{ in } \Omega.
\end{equation}
The problem \eqref{Eq:Yamabe}-\eqref{Eq:X1BC} is the so-called Loewner-Nirenberg problem (or, in the manifold setting, singular Yamabe problem) for negative curvature. The classical paper of Loewner and Nirenberg \cite{LoewnerNirenberg} shows that a solution exists if the  boundary of $\Omega$ is smooth and compact.

The next step is, for a fixed subset $\Sigma$ of $\partial\Omega$, to study equation \eqref{Eq:Yamabe}
under the new boundary condition
\begin{equation}
u(x) \rightarrow +\infty  \text{ as } \textrm{dist}(x,\Sigma) \rightarrow 0,
	\label{new-BC}
\end{equation}
If $\Sigma$ is a smooth compact submanifold of $\RR^n$ of dimension $d$ such that $\partial\Omega\setminus\Sigma$ is compact, \cite{LoewnerNirenberg} shows that, loosely speaking,  \eqref{new-BC} is not satisfied for $d < \frac{n-2}{2}$ and is satisfied for $d > \frac{n-2}{2}$. It was conjectured in \cite{LoewnerNirenberg} that \eqref{new-BC} is not satisfied in the equality case $d = \frac{n-2}{2}$, which was later confirmed in Aviles \cite{Aviles82-CPDE} and V\'eron \cite{Veron81-JDE}. The manifold version was considered in  Andersson, Chru\'sciel and Friedrich \cite{Andersson-Chrusciel-Friedrich}, Aviles and McOwen \cite{Aviles-McOwen88}, Mazzeo \cite{Mazzeo:singular-Yamabe}; in particular, a very detailed asymptotic expansion near $\Sigma$ was established in these works. A necessary and sufficient condition for the satisfaction of \eqref{new-BC} in terms of (non-linear) capacities when $\Sigma$ needs not be smooth was given in  \cite{Labutin:Yamabe}  based on his previous work \cite{Labutin:Wiener}.

(All these results in the negative case are in  contrast with the positive one, where the seminal work of Schoen and Yau \cite{S-Y} proves that a complete, conformally flat metric with constant positive scalar curvature must have singular set $\Sigma$ of Hausdorff dimension no larger than $\frac{n-2}{2}$, and this is sharp at least for smooth $\Sigma$, cf. Mazzeo and Pacard \cite{Mazzeo-Pacard:construction}.)

The present paper generalizes these results to the fully non-linear setting. We first consider the existence of a viscosity solution (see Definition \ref{Def:ViscositySolution} below)
for problem \eqref{Eq:X1}-\eqref{Eq:X1BC}, in the case that $\partial\Omega$ is smooth:

\begin{thm}\label{Thm:SmoothDom}
Let $n \geq 3$ and $(f,\Gamma)$ satisfy \eqref{02}-\eqref{f-homo}. Let $\Omega \subset \RR^n$ be a bounded domain with smooth boundary $\partial \Omega$. Then the problem \eqref{Eq:X1}-\eqref{Eq:X1BC} has a unique continuous viscosity solution. Moreover the solution is locally Lipschitz in $\Omega$ and
\[
\lim_{\textrm{dist}(x,\partial\Omega) \rightarrow 0} \textrm{dist}(x,\partial\Omega)^{\frac{n-2}{2}} u(x) = C(f,\Gamma) \in (0,\infty).
\]
\end{thm}

Following \cite{LoewnerNirenberg}, we would now like define a candidate \emph{maximal} solution $u_\Omega$ to \eqref{Eq:X1} for an arbitrary domain $\Omega$ in $\RR^n$. To this end, select an increasing sequence $\Omega_1 \Subset \Omega_2 \Subset \ldots $ of bounded subdomains of $\Omega$ with smooth boundaries $\partial\Omega_j$ which are hypersurfaces such that $\Omega = \cup \Omega_j$. By the above theorem, there exists a unique solution $u_j$ for every $j$. The function $u_\Omega$, constructed as the limit
$$u_\Omega := \lim_{j \rightarrow \infty} u_j,$$
is well defined, non-negative and belongs to $C^{0,1}_{loc}(\Omega)$. Furthermore, $u_\Omega$ is monotone in $\Omega$: If $\Omega \subset \tilde\Omega$, then $u_\Omega \geq u_{\tilde\Omega}$ in $\Omega$. See Section \ref{section:general-domains}. 

One has the dichotomy that either $u_\Omega > 0$ in $\Omega$ or $u_\Omega \equiv 0$ in $\Omega$; see Lemma \ref{Lem:16IV18-Dic}. If $\RR^n \setminus \bar\Omega \neq \emptyset$, then $u_\Omega > 0$ in $\Omega$. When $\RR^n \setminus \bar\Omega = \emptyset$, both situations can occur.
For example, when $\Omega = \RR^n \setminus \{0\}$ and $\Gamma \subset \Gamma_1$, then $u_\Omega \equiv 0$. \footnote{Indeed, by Corollary \ref{cor-last} below, $u_\Omega$ is bounded on $\RR^n$ and goes to zero at infinity. Since $\Gamma \subset \Gamma_1$, $u_\Omega$ is sub-harmonic in $\RR^n \setminus \{0\}$ and thus in $\RR^n$ as $u_\Omega$ is bounded from below. This implies by the maximum principle that $u_\Omega \le 0$ and so $u_\Omega \equiv 0$.} On the other hand, when $\Omega = \RR^n \setminus \SSphere^{n-k}$ and when the vector $v_k$ defined in \eqref{model} below belongs to $\Gamma$, $u_\Omega$ is positive  in $\Omega$ and gives rise to the standard metric on $\HH^{n-k+1} \times \SSphere^{k-1}$.

 It is not hard to see that if $u_\Omega \not\equiv 0$ then $u_\Omega$ is indeed the maximal positive solution of \eqref{Eq:X1} in $\Omega$.

It remains to understand the location where $u_\Omega$ diverges to infinity. 
For this, we introduce, as in \cite{LoewnerNirenberg}, the following notion:

\begin{Def}\label{def:regular}
A compact subset $\Sigma$ of $\partial\Omega$ is called \emph{regular} (for $(f,\Gamma)$) if $u_\Omega(x) \rightarrow +\infty$ as $x \rightarrow \Sigma$.
\end{Def}

Our next theorem is a direct analogue of \cite[Theorem 5]{LoewnerNirenberg}:

\begin{thm}\label{Thm:05IV18-T1}
Let $\Omega$ be a domain of $\RR^n$, $n \geq 3$, and $(f,\Gamma)$ satisfy \eqref{02}-\eqref{f-homo}. Let $\Sigma$ be a compact subset of $\partial\Omega$ such that $\partial\Omega \setminus \Sigma$ is also compact. Suppose that $u_\Omega \not\equiv 0$. Then $\Sigma$ is regular if and only if there are an open neighborhood $U$ of $\Sigma$ and a function $\psi \in C^2(U \cap \Omega)$ such that
\[
f(\lambda(-A^\psi)) \geq 1 \text{ in } U \cap \Omega \text{ and } \psi(x) \rightarrow + \infty \text{ as } x \rightarrow \Sigma.
\]
\end{thm}

Theorem \ref{Thm:05IV18-T1} and known dimensional estimates  for the Loewner-Nirenberg problem imply  the following corollary:

\begin{cor}\label{cor:irregular}
Let $\Omega$ be a domain of $\RR^n$, $n \geq 3$, and $(f,\Gamma)$ satisfy \eqref{02}-\eqref{f-homo}. Suppose that there is some constant $c_0 > 0$ such that
\begin{equation}
\lambda_1 + \ldots + \lambda_n \geq c_0 \text{ whenever } f(\lambda) \geq 1.
	\label{Eq:05IV18-E1}
\end{equation}
Let $\Sigma$ be a compact subset of $\partial\Omega$ such that $\partial\Omega \setminus \Sigma$ is also compact. If $\Sigma$ is regular for $(f,\Gamma)$ then it is regular for the Loewner-Nirenberg problem (i.e. for $(\sigma_1, \Gamma_1)$). In particular, if the Hausdorff dimension of $\Sigma$ is less than $\frac{n-2}{2}$, then $\Sigma$ is irregular for $(f,\Gamma)$.
\end{cor}

\begin{rem}
If $\Gamma$ is convex and $f$ is concave, then \eqref{Eq:05IV18-E1} holds.\\
\end{rem}

We finally deal with the case that $\Sigma$ is a smooth compact $(n-k)$-dimensional submanifold, $1 \leq k \leq n$. Following Loewner and Nirenberg \cite{LoewnerNirenberg}, our singular metric is modelled by the the metric coming from $\mathbb H^{n-k+1}\times \mathbb S^{k-1} \approx \RR^n \setminus \RR^{n-k}$ (which makes sense for $1 \leq k \leq n-1$).  Thus its conformal Hessian \eqref{conformal-Hessian}, denoted by $A_0$, has constant eigenvalues. More precisely,
\begin{equation}\label{model}
\lambda(-A_0)=(\underbrace{1, \ldots, 1}_{n - k + 1 \text{ entries}}, \underbrace{-1, \ldots, - 1}_{k-1 \text{ entries}}) =: v_k.
\end{equation}
This gives a solution to \eqref{Eq:X1}, after a suitable rescaling, if and only if $v_k \in \Gamma$, and in which case, the boundary $\RR^{n-k}$ is regular. 

Note that if $k' \geq k$, then $v_{k} - v_{k'} \in \bar\Gamma_n$, and so if $v_k \in \Gamma$, then $v_{k'} \in \Gamma$ for all $k' \leq k$. In particular, for $\Gamma = \Gamma_1$ and $\Gamma = \Gamma_2$, $v_k \in \Gamma$ if and only if $k < \frac{n-2}{2}$ and $k < \frac{n - \sqrt{n} + 2}{2}$, respectively. For other $\Gamma_\ell$ cones, the expression for the dividing dimension becomes more involved.

Our theorem states that the dimensional estimate in the above model is indeed the best one can expect: an $(n-k)$-dimensional submanifold $\Sigma$ is regular if $v_k \in \Gamma$ and is irregular if $v_k \notin \bar \Gamma$. (When $v_k \in \partial\Gamma$, we suspect, as for the Loewner-Nirenberg problem and as in the above model, that $\Sigma$ is irregular, but do not pursue this in the present paper.)

\begin{thm}\label{thm:smooth-Sigma}
Let $\Omega$ be a domain of $\RR^n$, $n \geq 3$, and $(f,\Gamma)$ satisfy \eqref{02}-\eqref{f-homo}. Suppose that a subset $\Sigma$ of $\partial \Omega$ is a smooth compact embedded $(n-k)$-dimensional submanifold of $\RR^n$, $1 \leq k \leq n$, and suppose that $\partial\Omega \setminus \Sigma$ is compact. Let $v_k$ be given by \eqref{model}.

\begin{itemize}
\item[\emph{i.}] If $v_k \in \Gamma$ and if $u_\Omega \not\equiv 0$, then $\Sigma$ is regular.

\item[\emph{ii.}] If $v_k \notin \bar\Gamma$
then $\Sigma$ is irregular. In fact, $u_\Omega$ is bounded near $\Sigma$.
\end{itemize}
\end{thm}

As an application of the above theorem we obtain the following result for isolated singularities:

\begin{cor}\label{cor-last}
Let $n \geq 3$ and $(f,\Gamma)$ satisfy \eqref{02}-\eqref{f-homo}. Suppose that $\Gamma \subset \Gamma_1$.
\begin{enumerate}[(a)]
\item If $u$ solves \eqref{Eq:X1} in a punctured ball $B_r(0) \setminus \{0\}$, then $u$ is bounded in $B_r(0)$.
\item If $u$ solves \eqref{Eq:X1} on $\RR^n \setminus B_r(0)$, then $u(x) = O(|x|^{2-n})$ as $|x| \rightarrow \infty$.
\end{enumerate}
\end{cor}

Let us make a remark for the positive curvature case. The available conditions for the existence of a singular metrics are far from being optimal. For instance, the first named author showed in \cite{Gonzalez05} that a singular metric with positive $\sigma_\ell$ must have a singular set $\Sigma$ of Hausdorff dimension less than $\frac{n-2\ell}{2}$, and established in \cite{Gonzalez06-RemSing} conditions for removability. On the other hand, by inspecting the model above e.g. for $\Gamma_2$, one expects that the sharp dimensional bound is
\begin{equation}\label{model-sigma2-positive}
\lambda(A_0)\in\Gamma_2\quad\text{iff}\quad k > \tfrac{n+\sqrt{n}+2}{2},
\end{equation}
Moreover, in the forthcoming paper \cite{Gonzalez-Mazzieri}, the authors show that for $\Sigma$ a submanifold, \eqref{model-sigma2-positive} seems to be sharp. 

In a related note for the positive curvature case, isolated singularities are much better understood and indeed they can be classified, see Chang, Han and Yang \cite{C-H-Y}, Gonz\'alez \cite{Gonzalez06, Gonzalez06-RemSing}, Han, Li and Teixeira \cite{HLT10}, Li \cite{Li06-JFA}, Ou \cite{Ou:singularities}. See also \cite{LiNgBocher} for related result, which in a sense can be viewed as the analogue in the ``zero'' case.\\

The existence statement in Theorem \ref{Thm:SmoothDom} is proved using Perron's method (cf. Ishi \cite{Ishii89-CPAM}) adapted to our fully non-linear equation. Here the essential step is the comparison principle from a joint work of the second and third named authors with B. Wang \cite{LiNgWang}. What is more delicate is the behavior of the solution near the boundary. As in \cite{LoewnerNirenberg}, we compare to the canonical solutions on suitable balls in order to obtain a-priori estimates: for instance in Theorem \ref{thm:Lip} we show that any viscosity solution must be locally Lipschitz.

In the particular case that $\Sigma$ is a smooth hypersurface, the boundary behavior is very precise (see Lemma \ref{Lem:BBehavior}). The proof of Theorem \ref{thm:smooth-Sigma} involves a very delicate choice of test function in Theorem \ref{Thm:05IV18-T1} and the structure of the cone $\Gamma$ is a crucial ingredient.

What it remains to do is to extend the classification results of Theorem \ref{thm:smooth-Sigma} to any
singular set $\Sigma\subset\partial\Omega$, not-necessarily smooth. We expect that a new notion of non-linear capacity would be an essential tool. Several definitions of capacity for $\sigma_\ell$ type problems have already been proposed in the literature (see \cite{Labutin:potential-estimates} and \cite{Gonzalez06-RemSing}) but it does not seem to be straightforward to adapt them for singular problems.

Another possible direction of research is the study of the geometry of a hypersurface in a conformal manifold $\Omega$. In the papers \cite{Graham:volume-renormalization,Gover-Waldron:renormalized-volume} the authors perform a volume renormalization procedure for a singular Yamabe metric in the negative curvature case, this is, for a solution to \eqref{Eq:Yamabe}-\eqref{Eq:X1BC}. In fact, this is the closest one can get to a Poincar\'e-Einstein manifold with a given conformal infinity and the process yields higher dimensional analogues of the Willmore energy. The main step in their proofs is to understand very precisely the asymptotic expansion of the boundary condition \eqref{new-BC}. It would be interesting to replace the Yamabe equation \eqref{Eq:Yamabe} by $\sigma_\ell$ in order to obtain different conformal information.\\

The structure of the paper is as follows:  Section \ref{Section:preliminaries} recalls some ingredients in the later proofs such as the definition of viscosity solutions, the canonical solution on balls and the comparison principle for fully-nonlinear equations. Section \ref{section:apriori} contains the proof of the $C^0$ and Lipschitz estimates, which, by Perron's method, yield the proof of Theorem \ref{Thm:SmoothDom}. Then, in Section \ref{section:general-domains} we turn to general domains and give the proof of Theorem \ref{Thm:05IV18-T1}. By finding a suitable test function in this theorem, we can give a precise characterization of the dimension of singular set when $\Sigma$ is a smooth submanifold and complete the proof of Theorem \ref{thm:smooth-Sigma}.

%====================%

\section{Preliminaries} \label{Section:preliminaries}

\subsection{Viscosity solutions}

For any set $S \subset\mathbb{R}^{n}$, we use $\mbox{USC}(S)$ to denote the set of functions $\psi:S\rightarrow\mathbb{R}\cup\{-\infty\}$, $\psi \not\equiv -\infty$ in $S$, satisfying
\begin{equation*}
\limsup\limits_{x\rightarrow\bar{x}}\psi(x)\leq \psi(\bar{x}),\quad \forall \bar{x}\in S.
\end{equation*}
Similarly, we use $\mbox{LSC}(S)$ to denote the set of functions $\psi: S\rightarrow\mathbb{R}\cup\{+\infty\}$, $\psi \not\equiv +\infty$  in $S$, satisfying
\begin{equation*}
\liminf\limits_{x\rightarrow\bar{x}}\psi(x)\geq \psi(\bar{x}),\quad \forall \bar{x}\in S.
\end{equation*}

\begin{Def}\label{Def:ViscositySolution}
Let $\Omega\subset\mathbb{R}^{n}$ be an open set. We say that a function  $u \in USC(\Omega)$ ($LSC(\Omega)$) is a sub-solution (super-solution) to \eqref{Eq:X1} in the viscosity sense, or alternatively
\begin{equation*}
f(\lambda(-A^u)) \geq 1 \quad \left(f(\lambda(-A^u)) \leq 1\right) \quad\mbox{in }\Omega,
\end{equation*}
if for any $x_{0}\in\Omega$, $\varphi\in C^{2}(\Omega)$, $(u-\varphi)(x_{0})=0$ and
\begin{equation*}
u-\varphi\leq0\quad(u-\varphi\geq0),\quad\mbox{near }x_{0},
\end{equation*}
there holds
\begin{equation*}
f(\lambda(-A^\varphi)) \geq 1 \quad \left(-A^\varphi(x_{0})\in \RR^n \setminus \bar\Gamma \text{ or } f(\lambda(-A^\varphi)) \leq 1\right).
\end{equation*}

We say that a function $u \in C^0(\Omega)$ satisfies \eqref{Eq:X1} in the viscosity sense if it is both a sub- and a super-solution to \eqref{Eq:X1} in the viscosity sense.
\end{Def}

\subsection{Comparison principle}

The following is a consequence of \cite[Theorem 3.2]{LiNgWang}.

\begin{prop}\label{Prop:CP}
Let $\Omega$ be a bounded domain of $\RR^n$, $n \geq 3$, and $(f,\Gamma)$ satisfy \eqref{02}-\eqref{f-homo}. Assume that $w \in LSC(\Omega)$ and $v \in USC(\Omega)$ are respectively a super-solution and a sub-solution of \eqref{Eq:X1} and assume that $w \geq v$ on $\partial\Omega$. Then $w \geq v$ in $\Omega$.
\end{prop}

\begin{proof}
Let $U$ be the set of symmetric $n \times n$ matrices $M$ such that 
\[
-\frac{1}{2}\lambda(M) \in \RR^n \setminus \{\mu \in\Gamma: f(\mu) \geq 1\},
\]
i.e. either $-\frac{1}{2}\lambda(M) \notin \bar\Gamma$, or $-\frac{1}{2}\lambda(M) \in \bar\Gamma$ and $f(-\frac{1}{2}\lambda(M)) < 1$. Then $U$ satisfies
\begin{align*}
M \in U, N \geq 0 &\Longrightarrow M + N \in U,\\
M \in U, c \in (0,1) &\Longrightarrow c\,M  \in U.
\end{align*}

Let $u = \psi^{-\frac{n-2}{4}}$. Then
\begin{align*}
A^u
	& = -\frac{2}{n-2} u^{-\frac{n+2}{n-2}}\,\nabla^2 u + \frac{2n}{(n-2)^2} u^{-\frac{2n}{n-2}} \nabla u \otimes \nabla u - \frac{2}{(n-2)^2}\,u^{-\frac{2n}{n-2}}\,|\nabla u|^2\,I\\
	&= \frac{1}{2} \nabla^2 \psi - \frac{1}{4} \psi^{-1} \nabla \psi \otimes \nabla \psi
	 - \frac{1}{8}\,\psi^{-1}\,|\nabla \psi|^2\,I =: \frac{1}{2}F[\psi].
\end{align*}
Then the pair $(F,U)$ satisfies the principle of propagation of touching points \cite[Theorem 3.2]{LiNgWang}, i.e. if $\psi_1$ and $\psi_2$ are such that $F[\psi_1] \notin U$, $F[\psi_2] \in \bar U$, $\psi_1 \geq \psi_2$ in $\Omega$ and $\psi_1 > \psi_2$ on $\partial\Omega$, then $\psi_1 > \psi_2$ in $\Omega$.

Assume by contradiction that $w - v$ is negative somewhere in $\Omega$. Then there is some $c > 1$ such that $c w \geq v$ in $\Omega$ and $cw - v$ vanishes somewhere in $\Omega$. Note that $cw$ is a super-solution of \eqref{Eq:X1} and $v$ is a sub-solution of \eqref{Eq:X1}. Hence $F[(cw)^{-\frac{n-2}{2}}] \in \bar U$ and $F[v^{-\frac{n-2}{2}}] \notin U$.  Hence, by the principle of propagation of touching points, $cw > v$ in $\Omega$, which contradicts our choice of $c$.
\end{proof}

\subsection{Canonical solutions on balls}

As in \cite{LoewnerNirenberg}, we will make use of the following solutions on balls and on exterior of balls, which come from the Poincar\'e metric: There exists a unique $\alpha = \alpha(f) > 0$ such that the functions
\[
u_{R,x_0}^{(in)}(x) = \alpha \Big(\frac{R}{R^2 - |x - x_0|^2}\Big)^{\frac{n-2}{2}}, \qquad u_{R,x_0}^{(out)}(x) = \alpha \Big(\frac{R}{|x - x_0|^2 - R^2}\Big)^{\frac{n-2}{2}}
\]
satisfy
\[
f(\lambda(-A^{u_{R,x_0}^{(in)}})) = 1 \text{ in } B_R(x_0) \text{ and } f(\lambda(-A^{u_{R,x_0}^{(out)}})) = 1 \text{ in } \RR^n \setminus B_R(x_0) .
\]
In fact, $\lambda(-A^{u_{R,x_0}^{(in)}})$ is constant in $B_R(x_0)$ and $\lambda(-A^{u_{R,x_0}^{(out)}})$ is constant $\RR^n \setminus B_R(x_0)$.

\begin{lem}
$u_{R,x_0}^{(in)}$ is the unique solution to \eqref{Eq:X1} for $\Omega = B_{R}(x_0)$.
\end{lem}

\begin{proof}
We write $B_R = B_R(x_0)$. Let $u$ be a solution to \eqref{Eq:X1} with $\Omega = B_R$. By the comparison principle,
\[
u \geq u_{R',x_0}^{(in)} \text{ in } B_R \text{ for all } R' > R
\]
and
\[
u \leq u_{R'',x_0}^{(in)} \text{ in } B_{R''} \text{ for all } 0 < R'' < R.
\]
Sending $R' \rightarrow R$ and $R'' \rightarrow R$, we obtain the conclusion.
\end{proof}

%====================================================================================

\section{A-priori estimates and solutions on bounded domains with smooth boundary}\label{section:apriori}

\subsection{$C^0$ estimates}

We start with giving an upper bound.
\begin{lem} \label{Lem:UpBnd}
Let $\Omega$ be a domain in $\RR^n$, $n \geq 3$, and $(f,\Gamma)$ satisfy \eqref{02}-\eqref{f-homo}. Assume that $0 < v \in USC(\Omega)$ is a sub-solution to \eqref{Eq:X1}. Then
\[
\textrm{dist}(x,\partial\Omega)^{\frac{n-2}{2}} v(x) \leq \alpha,
\]
where $\alpha$ is the constant in the expression for the canonical solution on balls.
\end{lem}

\begin{proof} Fix some $x \in \Omega$ and let $R = \textrm{dist}(x,\partial\Omega)$. By the comparison principle, we have
\[
v \leq u_{R, x}^{(in)} \text{ in } B_{R}(x).
\]
In particular, we have
\[
v(x) \leq u_{R, x}^{(in)}(x) = \alpha R^{-\frac{n-2}{2}},
\]
which implies the assertion.
\end{proof}

We turn to lower bounds. 

\begin{lem} \label{Lem:LowBndEZ}
Let $\Omega$ be a bounded domain in $\RR^n$, $n \geq 3$, and $(f,\Gamma)$ satisfy \eqref{02}-\eqref{f-homo}. If $0 < w \in LSC(\Omega)$ is a super-solution to \eqref{Eq:X1}-\eqref{Eq:X1BC}, and if $\RR^n \setminus \bar \Omega$ contains a ball $B(x_0, R)$, then 
\[
w \geq u_{R,x_0}^{(out)} > 0 \text{ in } \Omega.
\]
\end{lem}

\begin{proof}
The conclusion is an immediate consequence of the comparison principle.
\end{proof}

\begin{lem} \label{Lem:LowBnd}
Let $\Omega$ be a bounded domain in $\RR^n$, $n \geq 3$, and $(f,\Gamma)$ satisfy \eqref{02}-\eqref{f-homo}. Assume that $\partial\Omega$ is a smooth hypersurface. There exists some $c_0 = c_0(\Omega,f,\Gamma) > 0$ such that if $0 < w \in LSC(\Omega)$ is a super-solution to \eqref{Eq:X1} and that $w \geq c \geq c_0$ on $\partial\Omega$ for some constant $c$, then
\[
w(x) \geq  \frac{\alpha}{2^{\frac{n-2}{2}} [ (2 + \alpha^{-\frac{2}{n-2}} c_0^{\frac{2}{n-2}}\textrm{dist}(x,\partial\Omega) ) \textrm{dist}(x,\partial\Omega) + \alpha^{\frac{2}{n-2}}c^{-\frac{2}{n-2}}]^{\frac{n-2}{2}}},
\]
where $\alpha$ is the constant in the expression for the canonical solution on balls.
\end{lem}

\begin{proof}
Fix some $R_0 > 0$ such that, for every $\xi \in \partial\Omega$, there is some $x_0 = x_0(\xi) \in \RR^n$ such that $\bar B_{R_0}(x_0(\xi)) \cap \bar\Omega = \{\xi\}$.

We choose $c_0$ such that $\alpha^{\frac{2}{n-2}}c_0^{-\frac{2}{n-2}} = R_0$. Then, for $\epsilon = \frac{1}{2}\,\alpha^{\frac{2}{n-2}}c^{-\frac{2}{n-2}}  \leq \frac{1}{2}R_0$, there holds
\[
u_{R_0 - \epsilon, y}^{(out)}
	\leq \alpha\Big(\frac{R_0 - \epsilon}{2R_0 \epsilon - \epsilon^2}\Big)^{\frac{n-2}{2}}
	\leq \alpha\Big(\frac{1}{2\epsilon}\Big)^{\frac{n-2}{2}} =  c \text{ in } \RR^n \setminus B_{R_0}(y) \text{ for all } y \in \RR^n.
\]
Thus, by the comparison principle, we have
\[
w \geq u_{R_0 - \epsilon_0, x_0(\xi)}^{(out)}  \text{ in } \Omega \text{ for all } \xi \in \partial\Omega.
\]

Now, for any $x \in \Omega$, select a $\pi(x) \in \partial \Omega$ satisfying $|x - \pi(x)| = \textrm{dist}(x,\partial\Omega)$. We then have
\begin{align*}
w(x)
	&\geq u_{R_0 - \epsilon, x_0(\pi(x))}^{(out)}
	= \alpha\Big(\frac{R_0 - \epsilon}{(|x - \pi(x)| + R_0)^2 - (R_0 - \epsilon)^2}\Big)^{\frac{n-2}{2}}\\
	&= \alpha\Big(\frac{R_0 - \epsilon}{|x - \pi(x)|^2 + 2R_0(|x - \pi(x)|  + \epsilon) - \epsilon^2}\Big)^{\frac{n-2}{2}}\\
	&\geq \alpha\Big(\frac{R_0/2}{|x - \pi(x)|^2 + 2R_0(|x - \pi(x)|  + \epsilon) }\Big)^{\frac{n-2}{2}}.
\end{align*}
The assertion follows.
\end{proof}

The next lemma gives the boundary behavior of solutions.

\begin{lem}\label{Lem:BBehavior}
Let $\Omega$ be a bounded domain in $\RR^n$, $n \geq 3$, and $(f,\Gamma)$ satisfy \eqref{02}-\eqref{f-homo}. Assume that $\partial\Omega$ is a smooth hypersurface and $u$ is a solution to \eqref{Eq:X1}-\eqref{Eq:X1BC}. Then
\[
\lim_{\textrm{dist}(x,\partial\Omega)  \rightarrow 0} \textrm{dist}(x,\partial\Omega)^{\frac{n-2}{2}} u(x) = \alpha\,2^{-\frac{n-2}{2}},
\]
where $\alpha$ is the constant in the expression for the canonical solution on balls.
\end{lem}

\begin{proof}

For $\xi \in \partial \Omega$ and let $\nu(\xi)$ be the outward unit normal to $\partial\Omega$ at $\xi$. Fix some $R > 0$ such that $\bar B_R(\xi + R\nu(\xi)) \cap \bar\Omega = \{\xi\}$ and $\bar B_R(\xi - R\nu(\xi)) \cap (\RR^n \setminus\Omega) = \{\xi\}$ for all $\xi \in \partial\Omega$. By the comparison principle,
\[
u \geq u_{R,\xi + R\nu(\xi)}^{(out)}\text{ in } \Omega \text{ for all } \xi \in \partial\Omega.
\]

For $x$ close enough to $\partial\Omega$, let $\pi(x) \in \partial \Omega$ be the orthogonal projection of $x$ onto $\partial\Omega$. Then
\begin{align*}
\textrm{dist}(x,\partial\Omega)^{\frac{n-2}{2}} u(x)
	&= |x - \pi(x)|^{\frac{n-2}{2}} u(x)\\
	&\geq |x - \pi(x)|^{\frac{n-2}{2}} u_{R,\pi(x) + R\nu(\pi(x))}^{(out)}(x)\\
	&= \alpha \Big(\frac{R|x - \pi(x)|}{|x - \pi(x) - R\nu(\pi(x)|^2 - R^2}\Big)^{\frac{n-2}{2}}\\
	&= \alpha\Big(\frac{R|x - \pi(x)|}{(|x - \pi(x)| + R)^2 - R^2}\Big)^{\frac{n-2}{2}}\\
	&= \alpha\Big(\frac{R}{2R + |x - \pi(x)| }\Big)^{\frac{n-2}{2}}\\
	&= \alpha\Big(\frac{R}{2R + \textrm{dist}(x,\partial\Omega)}\Big)^{\frac{n-2}{2}}.
\end{align*}
It follows that
\[
\liminf_{x \rightarrow \xi} \textrm{dist}(x,\partial\Omega)^{\frac{n-2}{2}} u(x)  \geq \alpha\,2^{-\frac{n-2}{2}}.
\]

Similarly, we have
\[
u \leq u_{R,\xi - R\nu(\xi)}^{(in)}\text{ in } B_R(\xi - R\nu(\xi)) \text{ for all } \xi \in \partial\Omega,
\]
which implies
\begin{align*}
\textrm{dist}(x,\partial\Omega)^{\frac{n-2}{2}} u(x)
	&\leq |x - \pi(x)|^{\frac{n-2}{2}} u_{R,\pi(x) - R\nu(\pi(x))}^{(in)}(x)\\
	&= \alpha\Big(\frac{R}{2R - \textrm{dist}(x,\partial\Omega)}\Big)^{\frac{n-2}{2}},
\end{align*}
and so
\[
\limsup_{x \rightarrow \xi} \textrm{dist}(x,\partial\Omega)^{\frac{n-2}{2}} u(x)  \leq \alpha\,2^{-\frac{n-2}{2}}.
\]
The assertion follows.
\end{proof}

\subsection{Interior gradient estimates}

\begin{thm}\label{thm:Lip} Let $\Omega$ be a domain in $\RR^n$, $n \geq 3$, and $(f,\Gamma)$ satisfy \eqref{02}-\eqref{f-homo}. If $u$ is a continuous viscosity solution of \eqref{Eq:X1} in $\Omega$ satisfying $0 < a \leq u \leq b < +\infty$ for some constants $a, b$, then $u$ is locally Lipschitz in $\Omega$. Furthermore, there holds $\frac{|\nabla u|}{u} \leq C(\frac{b}{a},n,\Omega',\Omega)$ in $\Omega'$ for any $\Omega' \Subset \Omega$ .
\label{thm:regularity}
\end{thm}

\begin{proof} The proof is identical to that of \cite[Theorem 1.1]{LiNgWang}, which we reproduce here for readers' convenience. Without loss of generality, we may assume that $\Omega=B(0,1)$ and we only need to prove that $u$ is Lipschitz continuous on $B(0,\frac{1}{2})$.

For any $x\in\overline{B(0,\frac{1}{2})}$, $0<\lambda\leq R:=\frac{1}{4}\left[\frac{\sup\limits_{B(0,\frac{3}{4})}u}{\inf\limits_{B(0,\frac{3}{4})}u}\right]^{-\frac{1}{n-2}}$, we define $u_{x,\lambda}$, the Kelvin transform of $u$, as
\begin{equation*}
u_{x,\lambda}(y):=\frac{\lambda^{n-2}}{|y-x|^{n-2}}u\Big(x+\frac{\lambda^{2}(y-x)}{|y-x|^{2}}\Big),
\quad\forall y\in\overline{B\big(0,\tfrac{3}{4}\big)\setminus B(x,\lambda)}.
\end{equation*}
For any $y\in\partial B(0,\frac{3}{4})$, we have
\begin{equation*}
u_{x,\lambda}(y)\leq(4R)^{n-2}\sup\limits_{B(0,\frac{3}{4})}u=\inf\limits_{B(0,\frac{3}{4})}u\leq u(y).
\end{equation*}
By conformal invariance,
\begin{equation*}
f(\lambda(-A^{u_{x,\lambda}})) = 1 \quad\mbox{ in }B\big(0,\tfrac{3}{4}\big)\setminus \overline{B(x,\lambda)},\quad\mbox{ in the viscosity sense}.
\end{equation*}
Since $u_{x,\lambda}=u$ on $\partial B(x,\lambda)$, the comparison principle gives
\begin{equation*}
u_{x,\lambda}\leq u\mbox{ in }B\big(0,\tfrac{3}{4}\big)\setminus \overline{B(x,\lambda)}\mbox{ for any }0<\lambda\leq R,x\in\overline{B\big(0,\tfrac{1}{2}\big)}.\label{lashi}
\end{equation*}

By \cite[Lemma 2]{LiNg-arxiv}, \eqref{lashi} implies that $u$ is Lipschitz continuous on $\overline{B(0,\frac{1}{2})}$. This concludes the proof.
\end{proof}

\subsection{Solutions on bounded smooth domains}

In this subsection, we assume that $\Omega$ is a bounded domain whose boundary $\partial\Omega$ is a smooth compact hypersurface with a finite number of connected components.

\begin{lem}\label{Lem:BCc}
Let $n \geq 3$ and $(f,\Gamma)$ satisfy \eqref{02}-\eqref{f-homo}. Let $\Omega \subset \RR^n$ be a bounded domain with smooth boundary $\partial \Omega$. For $c \in (0,\infty)$, there is a unique continuous viscosity solution to \eqref{Eq:X1} subjected to the boundary condition $u(x) = c$ on $\partial \Omega$.
\end{lem}

\begin{proof} The uniqueness is a consequence of the comparison principle and the following two facts:
\begin{itemize}
\item If $w$ is a super-solution of \eqref{Eq:X1} and $t \geq 1$, then $tw$ is also a super-solution of \eqref{Eq:X1}.
\item If $v$ is a sub-solution of \eqref{Eq:X1} and $0 < t \leq 1$, then $tv$ is also a sub-solution of \eqref{Eq:X1}.
\end{itemize}

We prove the existence by Perron's method. By the smoothness of $\partial\Omega$, there exists some $R_0 > 0$ such that, for any $\xi \in \partial \Omega$, there is some $x_0 = x_0(\xi) \in \RR^n \setminus \Omega$ such that $\bar B_{R_0}(x_0) \cap \bar \Omega = \{\xi\}$. Also, there exists $\epsilon_0 = \epsilon_0(R_0,c) \in (0,R_0)$ such that
\[
u_{R_0 - \epsilon_0,y}^{(out)} = c \text{ on } \partial B_{R_0}(y) \text{ for all } y \in \RR^n.
\]
It follows that the functions $\omega^{(\xi)} := u_{R_0 - \epsilon_0,x_0(\xi)}^{(out)}$ belong to $C^\infty(\bar\Omega)$, satisfy
\[
\sigma_k(\lambda(-A^{\omega^{(\xi)}})) = 1 \text{ in }\Omega,
\]
and
\[
\omega^{(\xi)}(\xi) = c \text{ and } \omega^{(\xi)} < c \text{ in } \bar\Omega \setminus \{\xi\}.
\]

For $x \in \bar\Omega$, define
\begin{align}
\underline{u}(x)
	&= \sup\Big\{\omega^{(\xi)}(x): \xi \in \partial\Omega\Big\} \leq c.
\end{align}
It is clear that $\underline{u} \equiv c$ on $\partial\Omega$.

Let
\[
K = \sup_{\xi \in \partial\Omega} \sup_{\bar\Omega} |\nabla \omega^{(\xi)}|.
\]
For any $x, y \in \bar\Omega$ and $\xi \in \partial\Omega$, we have
$$\omega^{(\xi)}(x) \leq \omega^{(\xi)}(y) + K|x - y| \leq \underline{u}(y) + K|x - y|,
$$
which implies that $\underline{u}(x) \leq \underline{u}(y) + K|x - y|$. This shows that $\underline{u}$ is Lipschitz continuous in $\bar\Omega$. By a standard argument (see \cite{UserGuide} and \cite[Section 4]{LiNgWang}), $\underline{u}$ is a viscosity subsolution of \eqref{Eq:X1}.

On the other hand, the constant function $c$ is a supersolution to \eqref{Eq:X1}.

By Perron's method (see \cite[Section 4 and Remark 4.2]{LiNgWang}), \eqref{Eq:X1} has a unique solution satisfying $u \equiv c$ on $\partial\Omega$.
\end{proof}

\medskip

\begin{proof}[Proof of Theorem \ref{Thm:SmoothDom}]
 Uniqueness is a consequence of the comparison principle and the boundary estimate Lemma \ref{Lem:BBehavior}.

We turn to existence. Let $u_c$ be the solution to \eqref{Eq:X1} satisfying $u_c \equiv c$ on $\partial \Omega$ constructed in Lemma \ref{Lem:BCc}.

By the comparison principle, $u_c$ is monotonically increasing in $c$. By Lemma \ref{Lem:UpBnd}, $u_c$ is uniformly bounded from above. Let
\[
u(x) = \lim_{c \rightarrow \infty} u_c(x) \text{ for } x \in \Omega.
\]
By Lemma \ref{Lem:LowBnd}, we have
\[
u(x) \geq \frac{1}{C}\,\textrm{dist}(x,\partial\Omega)^{-\frac{n-2}{2}}
\]
for some positive constant $C$ and for all $x$ sufficiently close to $\partial\Omega$. In particular, \eqref{Eq:X1BC} is satisfied.

By Lemmas \ref{Lem:UpBnd}, \ref{Lem:LowBnd} and the Lipschitz regularity result Theorem \ref{thm:Lip}, for any $\Omega' \Subset \Omega$, there is some $c_0 > 0$ and $C > 0$such that
\[
|\nabla u_c| \leq C \text{ in } \Omega' \text{ for all } c > c_0.
\]
It follows that $u_c$ converges uniformly on compact subsets of $\Omega$ to $u \in C^{0,1}_{loc}(\Omega)$. By a standard argument, $u$ is a viscosity solution of \eqref{Eq:X1}.  The blow-up rate of $u$ near $\partial\Omega$ follows from Lemma \ref{Lem:BBehavior}. This concludes the proof.
\end{proof}

\section{Maximal solutions on general domains}\label{section:general-domains}

We would now like to construct a candidate maximal solution to \eqref{Eq:X1} for an arbitrary domain $\Omega$ in $\RR^n$. By performing an inversion about a sphere we may assume without loss of generality that $\partial\Omega$ is compact. Let $\Omega_1 \Subset \Omega_2 \Subset \ldots $ be an increasing sequence of bounded subdomains of $\Omega$ with smooth boundaries $\partial\Omega_j$ such that $\Omega = \cup \Omega_j$. By Theorem \ref{Thm:SmoothDom}, \eqref{Eq:X1}-\eqref{Eq:X1BC} has a unique solution $u_j$ on $\Omega_j$.

By the comparison principle, $u_j$ is non-increasing. Furthermore, by Lemmas \ref{Lem:UpBnd}, \ref{Lem:LowBndEZ} and the Lipschitz regularity result Theorem \ref{thm:Lip}, if $\RR^n \setminus \bar \Omega$ is non-empty, then $\ln u_j$ are locally uniformly bounded in $C^{0,1}_{loc}$. (See also Corollary \ref{Cor:21IV17-C1} below.)

We define (cf. \cite{LoewnerNirenberg}):

\begin{Def}
Set $u_\Omega = \lim_{j \rightarrow \infty} u_j \ge 0$.
\end{Def}

It is clear that $u_\Omega$ is well defined, independent of the choice of $\Omega_j$. In addition, if $\Omega \subset \tilde\Omega$, then $u_\Omega \geq u_{\tilde\Omega}$ in $\Omega$. If $u_\Omega > 0$ (which is the case if e.g. $\RR^n \setminus \bar\Omega \neq\emptyset$; see Lemma \ref{Lem:LowBndEZ}), it is the maximal positive solution of \eqref{Eq:X1}.  Furthermore, if $\Omega$ is bounded and $\partial\Omega$ is a hypersurface, then $u_\Omega$ coincides with the solution given by Theorem \ref{Thm:SmoothDom}. 

We have the following dichotomy:

\begin{lem}\label{Lem:16IV18-Dic}
Let $\Omega$ be a domain of $\RR^n$, $n \geq 3$, and $(f,\Gamma)$ satisfy \eqref{02}-\eqref{f-homo}. Then either $u_\Omega > 0$ in $\Omega$ or $u_\Omega \equiv 0$ in $\Omega$.
\end{lem}

\begin{proof} Suppose that $u_\Omega \not\equiv 0$ in $\Omega$. Then there is a ball $B(x_0,r_0) \subset \Omega$ and a constant $c > 0$ such that $u_\Omega > 2c$ in $B(x_0,r_0)$.

Select an exhaustion $\Omega_1 \Subset \Omega_2 \Subset \ldots $ of $\Omega$ and let $u_j$ denote the corresponding sequence as above. As $u_j$ is non-increasing and converges to $u_\Omega$, we have $u_j > c$ in $B(x_0, r_0)$ for all large $j$. Now, select $0 < r_1 < r_0$ such that
\[
u_{r_1, x_0}^{(out)} = \alpha \Big(\frac{r_1}{r_0^2 - r_1^2}\Big)^{\frac{n-2}{2}} = c \text{ on } \partial B(x_0,r_0).
\]
Then the comparison principle implies that 
\[
u_j \geq u_{r_1, x_0}^{(out)} \text{ in } \Omega_j \setminus B(x_0,r_0) \text{ for all large $j$},
\]
from which the conclusion follows.
\end{proof}

\begin{cor}\label{Cor:21IV17-C1}
Let $\Omega$ be a domain of $\RR^n$, $n \geq 3$, and $(f,\Gamma)$ satisfy \eqref{02}-\eqref{f-homo}. Then $u_\Omega$ belongs to $C^{0,1}_{loc}(\Omega)$.
\end{cor}

\begin{proof}
If $u_\Omega$ is identically zero, there is nothing to prove. Otherwise, we have $u_\Omega > 0$ by Lemma \ref{Lem:16IV18-Dic} above. 

Now recall the approximate solutions $u_j$ defines on subdomains $\Omega_j$. By construction, $u_j \geq u_\Omega > 0$ in $\Omega_j$. Thus, by Lemma \ref{Lem:UpBnd} and the Lipschitz regularity result Theorem \ref{thm:Lip}, the sequence $\ln u_j$ is locally uniformly bounded in $C^{0,1}_{loc}(\Omega)$, from which the conclusion is drawn.
\end{proof}

The rest of this section discusses the validity of \eqref{Eq:X1BC}. Recall that a compact subset $\Sigma$ of $\partial\Omega$ is called regular if $u_\Omega(x) \rightarrow +\infty$ as $x \rightarrow \Sigma$.

\subsection{A criterion for regularity}

\begin{proof}[Proof of Theorem \ref{Thm:05IV18-T1}]
 The proof is very similar to that given in \cite{LoewnerNirenberg}. It suffices to show that if $\psi$ exists, then $u(x) = u_\Omega(x) \rightarrow +\infty$ as $x \rightarrow \Sigma$.

We may assume that $\bar U$ is compact, $U \cap (\partial\Omega \setminus \Sigma)$ is empty and $\psi \in C^2(\bar U \cap \Omega)$. For large $M$, let
\[
U_M = \{x \in U \cap \Omega: \psi(x) \geq M\} \subset U.
\]
Let
\[
a_M = \inf_{\partial U_M} u > 0.
\]

Note that, for large $j$, $\partial(U_M \cap \Omega_j) = (\partial U_M \cap \Omega) \cup (\partial\Omega_j \cap U)$, where $\psi = M$ on the former set and $u_j = \infty$ on the latter set. It follows that
$
(1 + a_M^{-1}\,M) u_j \geq \psi \text{ on } \partial (U_M \cap \Omega_j)$. As $(1 + a_M^{-1}\,M) u_j$ is a super-solution to \eqref{Eq:X1}, the comparison principle then implies that
\[
(1 + a_M^{-1}\,M) u_j \geq \psi \text{ on } U_M \cap \Omega_j.
\]
Sending $j \rightarrow \infty$, we deduce that $u(x) \rightarrow +\infty$ as $x \rightarrow \Sigma$. 
\end{proof}

\begin{proof}[Proof of Corollary \ref{cor:irregular}]
Suppose that $\Sigma$ is regular and let $\psi$ be as in Theorem \ref{Thm:05IV18-T1}. By \eqref{Eq:05IV18-E1}, we have
\[
\Delta \psi - \psi^{\frac{n+2}{n-2}} \geq \frac{n-2}{2} c_0 > 0.
\]
By \cite[Theorem 5]{LoewnerNirenberg}, this implies that $\Sigma$ is regular for the Loewner-Nirenberg problem. By \cite[Theorem 7]{LoewnerNirenberg}, this implies further $\dim\Sigma \geq \frac{n-2}{2}$.
\end{proof}

\subsection{Proof of Theorem \ref{thm:smooth-Sigma}}

\begin{proof}[Proof of Theorem \ref{thm:smooth-Sigma}]
We first show statement \emph{i.}, by constructing a suitable test function in Theorem \ref{Thm:05IV18-T1}.
Let $\rho$ denote the distance function to $\Sigma$ and
\[
\psi = c\,\rho^{-\frac{n-2}{2}},
\]
where $c$ is a constant which will be fixed later. We have, as $|\nabla \rho| = 1$,
\[
A^\psi = c^{-\frac{4}{n-2}}\Big[\rho \nabla^2 \rho - \frac{1}{2} I\Big].
\]

There exists some $\delta > 0$ such that, for $\rho(x) < \delta$, there exists a unique point $\pi(x)$ on $\Sigma$ such that $\rho(x) = |x - \pi(x)|$ (and that $x - \pi(x)$ is perpendicular to $\Sigma$ at $\pi(x)$). Furthermore, the map $x \mapsto \pi(x)$ is smooth.

Fix some point $x_0 \in \Omega$ with $\rho(x_0) = \rho_0 < \delta$ and Assume, after a rotation of coordinate system, that $\pi(x_0) = 0$, $x_0 = (0, 0, \ldots, 0, \rho_0)$ and
$\Sigma$ is represented locally by
\[
\Sigma = \{(x', x'') \in  \RR^{n-k} \times \RR^{k}: x'' = f(x')\},
\]
where $f = (f_1, \ldots, f_k)$ maps a neighborhood of the origin in $\RR^{n-k}$ into $\RR^k$ with $f(0) = 0$ and $\nabla f(0) = 0$.

Let $e_1, \ldots, e_n$ be the standard basis of $\RR^n$.

Observe that for all $x = (0,x'') \in \{0\} \times \RR^k$ with $|x''| < \delta$, $\pi(x) = 0$. Also, we have $\pi(x', f(x')) = x'$. It follows that $\partial_i \pi(0) = e_i$ if $1 \leq i \leq n-k$ and $\partial_i \pi(0) = 0$ if $n - k + 1 \leq i \leq n$, i.e.
\[
\nabla \pi(0)  = \left[\begin{array}{cc}
	I_{(n-k) \times (n-k)} & 0_{(n-k) \times k} \\
	0_{k \times (n-k)} & 0_{k \times k}
\end{array}\right].
\]
Hence,
\begin{equation}
\pi(x) = x' + O(|x|^2)
	\label{Eq:DPiX}
\end{equation}
Note that $x - \pi(x)$ belongs to the span of
\[
\Big\{e_{n-k + i} - \sum_{l = 1}^{n-k} \partial_{l} f^i(\pi(x)')\,e_l: i = 1,\ldots, k\Big\},
\]
and so
\begin{align*}
x - \pi(x)
	&= \sum_{i=1}^k (x- \pi(x))_{n-k+i}\Big(e_{n-k + i} - \sum_{l = 1}^{n-k} \partial_{l} f^i(\pi(x))\,e_l\Big)\\
	&= (-(x'' - \pi(x)'') \nabla f(\pi(x)'),x'' - \pi(x)'').
\end{align*}
It follows that, in view of \eqref{Eq:DPiX}.
\[
\rho(x)^2 = |x'' - \pi(x)''|^2 + |(x'' - \pi(x)'') \nabla f(\pi(x)')|^2 = |x''|^2 + O(|x|^3).
\]
This implies that
\[
\rho(x_0) \nabla^2 \rho(x_0) = \left[\begin{array}{ccc}
	0_{(n-k) \times (n-k)} & 0_{(n-k) \times (k-1)} & 0_{(n-k) \times 1}\\
	0_{(k-1) \times (n-k)} & I_{(k-1) \times (k-1)} & 0_{(k-1) \times 1}\\
	0_{1 \times (n-k)} 	& 0_{1 \times (k-1)} & 0
\end{array}\right] + O(\rho_0).
\]
Hence, as $f$ is continuous, by shrinking $\delta$ if necessary, we have
\[
f\Big(\lambda\Big(\rho\nabla^2\rho - \frac{1}{2}I\Big)\Big) > \frac{1}{C} > 0 \text{ provided } \rho(x) < \delta.
\]
Hence using the fact that $f$ is positively homogeneous, we can find some $c$ such that $f(\lambda(-A^\psi)) \geq 1$ as desired.\\

Next we show statement \emph{ii.} For positive constants $\alpha$, $\beta$, $c$ and $d$, consider now the function
\[
\psi = \psi_{c,d} = (c\rho^{-\alpha} + d)^{\beta}.
\]
We have, as $|\nabla \rho| = 1$,
\begin{align*}
A^\psi
	&=  \frac{2\alpha\beta}{(n-2)\rho^2} \psi^{-\frac{4}{n-2}}\,\frac{c\rho^{-\alpha}}{c\rho^{-\alpha} + d} \Big\{\rho\,\nabla^2 \rho\\
		&\qquad  + \Big[-\alpha - 1 + \alpha(\frac{2}{n-2}\beta + 1) \frac{c\rho^{-\alpha}}{c\rho^{-\alpha} + d}\Big] \nabla \rho \otimes \nabla \rho\\
		&\qquad - \frac{\alpha\beta}{n-2} \frac{c\rho^{-\alpha}}{c\rho^{-\alpha} + d}\,I\Big\}.
\end{align*}
The previous calculation shows that, in an appropriate coordinate system and when $\rho$ is sufficiently small,
\begin{align*}
A^\psi &= \frac{2\alpha\beta}{(n-2)\rho^2} \psi^{-\frac{4}{n-2}}\,\zeta\times \\
	&\qquad \times \left[\begin{array}{ccc}
	-\zeta I_{(n-k) \times (n-k)} & 0_{(n-k) \times (k-1)} & 0_{(n-k) \times 1}\\
	0_{(k-1) \times (n-k)} & (1 - \zeta + O(\rho))I_{(k-1) \times (k-1)} & 0_{(k-1) \times 1}\\
	0_{1 \times (n-k)} 	& 0_{1 \times (k-1)} &  - \alpha - 1 + \zeta + \frac{n-2}{\beta}\zeta
\end{array}\right],
\end{align*}
where $\zeta = \frac{\alpha\beta}{n-2}\frac{c\rho^{-\alpha}}{c\rho^{-\alpha} + d}$.

Writing
\begin{align*}
&\frac{1}{1-\zeta}(\underbrace{\zeta, \ldots, \zeta}_{n - k \text{ entries}}, \underbrace{-(1-\zeta), \ldots, - (1-\zeta)}_{k-1 \text{ entries}}, - \alpha - 1 + \zeta + \frac{n-2}{\beta}\zeta)\\
	&\qquad = (\underbrace{1, \ldots, 1}_{n - k  \text{ entries}}, \underbrace{-1, \ldots, - 1}_{k-1 \text{ entries}}, 1)\\
		&\qquad\qquad + \Big(\underbrace{\frac{2\zeta - 1}{1 - \zeta}, \ldots, \frac{2\zeta - 1}{1 - \zeta}}_{n - k \text{ entries}}, \underbrace{0, \ldots, 0}_{k-1 \text{ entries}}, \frac{\alpha - \frac{n-2}{\beta}\zeta}{1-\zeta}\Big),
\end{align*}
we see that there exists some $\epsilon > 0$ such that the above vector does not belong to $\Gamma$ for all $\zeta \in [0,\frac{1}{2} + \epsilon]$ and $\alpha \in (0,\epsilon)$. Note that we have used the assumption that $k > k(\Gamma)$. 

We now fix some $\alpha \in (0,\epsilon)$ and some $\beta > 0$ such that $\frac{\alpha\beta}{n-2} \in (\frac{1}{2},\frac{1}{2} + \epsilon)$. By the above, there is some $\delta > 0$ such that
\[
\lambda(-A^{\psi_{c,d}}) \in \RR^n\setminus \bar\Gamma \text{ in } \{0 < \rho < \delta\}
\]
for all positive constants $c$ and $d$. Note that this implies in particular that $\psi_{c,d}$ is a super-solution of \eqref{Eq:X1}.

Now consider the function $u = u_\Omega$. Take $d$ sufficiently large such that $\psi_{c,d} \geq d^\beta > u$ on $\{\rho = \delta\}$. Observe that $\psi_{c,d} \geq c\rho^{-\alpha\beta}$, while, by Lemma \ref{Lem:UpBnd} and our assumption that $\frac{\alpha\beta}{n-2} > \frac{1}{2}$, $u = O(\rho^{-\frac{n-2}{2}}) = o(\rho^{-\alpha\beta})$ as $\rho \rightarrow 0$. Hence, there is some $\delta' = \delta'\in (0,\delta)$ such that $\psi_{c,d} > u$ in $\{0 < \rho < \delta'\}$. By the comparison principle,  we hence have $\psi_{c,d} > u$ in $\{\delta' < \rho < \delta\} \cap \{u > 0\}$ and so
\[
\psi_{c,d} \geq u \text{ in } \{0 < \rho < \delta\}.
\]
Since this is true for all $c$, we deduce that
\[
u \leq \lim_{c \rightarrow 0} \psi_{c,d} = d^\beta \text{ in } \{0 < \rho < \delta\}.
\]
In particular, $u$ is bounded near $\Sigma$.
\end{proof}

\begin{proof}[Proof of Corollary \ref{cor-last}]
It suffices to consider (a), as (b) is a consequence of (a) via an inversion. Furthermore, as $u_{B_r(0) \setminus \{0\}}$ is the maximal solution in $B_r(0) \setminus \{0\}$, it suffices to establish the result for $u = u_{B_r(0) \setminus \{0\}}$.

We have
\[
(1, \underbrace{-1, \ldots -1}_{n - 1 \text{ entries}}) \in \RR^n \setminus \bar\Gamma_1 \subset \RR^n \setminus \bar\Gamma.
\]
The conclusion follows from Theorem \ref{thm:smooth-Sigma} \emph{ii.} with $\Sigma = \{0\}$.
\end{proof}

%====================%

\noindent\textbf{Acknowledgements.} M.d.M. Gonz\'alez is supported by Spanish government grants MTM2014-52402-C3-1-P and MTM2017-85757-P, and the BBVA foundation grant for  Researchers and Cultural Creators, 2016. Y.Y. Li is partially supported by NSF grant DMS-1501004. The first two named authors are grateful to the Fields Institute in Toronto for hospitality during part of work on this paper.

\newcommand{\noopsort}[1]{}


\begin{thebibliography}{10}

\bibitem{Andersson-Chrusciel-Friedrich}
{\sc L.~Andersson, P.~T. Chru\'sciel, and H.~Friedrich}, {\em On the regularity
  of solutions to the {Y}amabe equation and the existence of smooth
  hyperboloidal initial data for {E}instein's field equations}, Comm. Math.
  Phys., 149 (1992), pp.~587--612.

\bibitem{Aviles82-CPDE}
{\sc P.~Aviles}, {\em A study of the singularities of solutions of a class of
  nonlinear elliptic partial differential equations}, Comm. Partial
  Differential Equations, 7 (1982), pp.~609--643.

\bibitem{Aviles-McOwen88}
{\sc P.~Aviles and R.~C. McOwen}, {\em Complete conformal metrics with negative
  scalar curvature in compact {R}iemannian manifolds}, Duke Math. J., 56
  (1988), pp.~395--398.

\bibitem{CGY02-AnnM}
{\sc S.-Y.~A. Chang, M.~J. Gursky, and P.~Yang}, {\em An equation of
  {M}onge-{A}mp\`ere type in conformal geometry, and four-manifolds of positive
  {R}icci curvature}, Ann. of Math. (2), 155 (2002{\noopsort{a}}),
  pp.~709--787.

\bibitem{CGY03-IHES}
\leavevmode\vrule height 2pt depth -1.6pt width 23pt, {\em A conformally
  invariant sphere theorem in four dimensions}, Publ. Math. Inst. Hautes
  \'Etudes Sci.,  (2003), pp.~105--143.

\bibitem{CGY03-IP}
\leavevmode\vrule height 2pt depth -1.6pt width 23pt, {\em Entire solutions of
  a fully nonlinear equation}, in Lectures on partial differential equations,
  vol.~2 of New Stud. Adv. Math., Int. Press, Somerville, MA, 2003, pp.~43--60.

\bibitem{C-H-Y}
{\sc S.-Y.~A. Chang, Z.-C. Han, and P.~Yang}, {\em Classification of singular
  radial solutions to the {$\sigma_k$} {Y}amabe equation on annular domains},
  J. Differential Equations, 216 (2005), pp.~482--501.

\bibitem{UserGuide}
{\sc M.~G. Crandall, H.~Ishii, and P.-L. Lions}, {\em User's guide to viscosity
  solutions of second order partial differential equations}, Bull. Amer. Math.
  Soc. (N.S.), 27 (1992), pp.~1--67.

\bibitem{GeWang06}
{\sc Y.~Ge and G.~Wang}, {\em On a fully nonlinear {Y}amabe problem}, Ann. Sci.
  \'Ecole Norm. Sup. (4), 39 (2006), pp.~569--598.

\bibitem{Gonzalez05}
{\sc M.~d.~M. Gonz{\'a}lez}, {\em Singular sets of a class of locally
  conformally flat manifolds}, Duke Math. J., 129 (2005), pp.~551--572.

\bibitem{Gonzalez06}
\leavevmode\vrule height 2pt depth -1.6pt width 23pt, {\em Classification of
  singularities for a subcritical fully nonlinear problem}, Pacific J. Math.,
  226 (2006), pp.~83--102.

\bibitem{Gonzalez06-RemSing}
\leavevmode\vrule height 2pt depth -1.6pt width 23pt, {\em Removability of
  singularities for a class of fully non-linear elliptic equations}, Calc. Var.
  Partial Differential Equations, 27 (2006), pp.~439--466.

\bibitem{Gonzalez-Mazzieri}
{\sc M.~d.~M. Gonz\'alez and L.~Mazzieri}, {\em Construction of singular
  metrics for a fully non-linear equation in conformal geometry}.
\newblock In preparation.

\bibitem{Gover-Waldron:renormalized-volume}
{\sc A.~R. Gover and A.~Waldron}, {\em Renormalized volume}, Comm. Math. Phys.,
  354 (2017), pp.~1205--1244.

\bibitem{Graham:volume-renormalization}
{\sc C.~R. Graham}, {\em Volume renormalization for singular {Y}amabe metrics},
  Proc. Amer. Math. Soc., 145 (2017), pp.~1781--1792.

\bibitem{Guan:negative-Ricci}
{\sc B.~Guan}, {\em Complete conformal metrics of negative {R}icci curvature on
  compact manifolds with boundary}, Int. Math. Res. Not. IMRN,  (2008),
  pp.~Art. ID rnn 105, 25.

\bibitem{GW03-IMRN}
{\sc P.~Guan and G.~Wang}, {\em Local estimates for a class of fully nonlinear
  equations arising from conformal geometry}, Int. Math. Res. Not.,  (2003),
  pp.~1413--1432.

\bibitem{Gursky-Streets-Warren}
{\sc M.~Gursky, J.~Streets, and M.~Warren}, {\em Existence of complete
  conformal metrics of negative {R}icci curvature on manifolds with boundary},
  Calc. Var. Partial Differential Equations, 41 (2011), pp.~21--43.

\bibitem{Gursky-Viaclovsky:negative-curvature}
{\sc M.~J. Gursky and J.~A. Viaclovsky}, {\em Fully nonlinear equations on
  {R}iemannian manifolds with negative curvature}, Indiana Univ. Math. J., 52
  (2003), pp.~399--419.

\bibitem{GV07}
\leavevmode\vrule height 2pt depth -1.6pt width 23pt, {\em Prescribing
  symmetric functions of the eigenvalues of the {R}icci tensor}, Ann. of Math.
  (2), 166 (2007), pp.~475--531.

\bibitem{HLT10}
{\sc Z.-C. Han, Y.~Y. Li, and E.~V. Teixeira}, {\em Asymptotic behavior of
  solutions to the {$\sigma_k$}-{Y}amabe equation near isolated singularities},
  Invent. Math., 182 (2010), pp.~635--684.

\bibitem{Ishii89-CPAM}
{\sc H.~Ishii}, {\em On uniqueness and existence of viscosity solutions of
  fully nonlinear second-order elliptic {PDE}s}, Comm. Pure Appl. Math., 42
  (1989), pp.~15--45.

\bibitem{Labutin:Yamabe}
{\sc D.~Labutin}, {\em Thinness for scalar-negative singular {Y}amabe metrics}.
\newblock Preprint, 2005; arXiv:math/0506226.

\bibitem{Labutin:potential-estimates}
{\sc D.~A. Labutin}, {\em Potential estimates for a class of fully nonlinear
  elliptic equations}, Duke Math. J., 111 (2002), pp.~1--49.

\bibitem{Labutin:Wiener}
\leavevmode\vrule height 2pt depth -1.6pt width 23pt, {\em Wiener regularity
  for large solutions of nonlinear equations}, Ark. Mat., 41 (2003),
  pp.~307--339.

\bibitem{LiLi03}
{\sc A.~Li and Y.~Y. Li}, {\em On some conformally invariant fully nonlinear
  equations}, Comm. Pure Appl. Math., 56 (2003), pp.~1416--1464.

\bibitem{LiLi05}
\leavevmode\vrule height 2pt depth -1.6pt width 23pt, {\em On some conformally
  invariant fully nonlinear equations. {II}. {L}iouville, {H}arnack and
  {Y}amabe}, Acta Math., 195 (2005), pp.~117--154.

\bibitem{Li-Sheng:flow}
{\sc J.~Li and W.~Sheng}, {\em Deforming metrics with negative curvature by a
  fully nonlinear flow}, Calc. Var. Partial Differential Equations, 23 (2005),
  pp.~33--50.

\bibitem{Li06-JFA}
{\sc Y.~Y. Li}, {\em Conformally invariant fully nonlinear elliptic equations
  and isolated singularities}, J. Funct. Anal., 233 (2006), pp.~380--425.

\bibitem{Li09-CPAM}
\leavevmode\vrule height 2pt depth -1.6pt width 23pt, {\em Local gradient
  estimates of solutions to some conformally invariant fully nonlinear
  equations}, Comm. Pure Appl. Math., 62 (2009), pp.~1293--1326.
\newblock (C. R. Math. Acad. Sci. Paris 343 (2006), no. 4, 249--252).

\bibitem{LiNgPoorMan}
{\sc Y.~Y. Li and L.~Nguyen}, {\em A compactness theorem for fully nonlinear
  {Y}amabe problem under a lower {R}icci curvature bound}, J. Funct. Anal., 266
  (2014), pp.~2741--3771.

\bibitem{LiNgBocher}
\leavevmode\vrule height 2pt depth -1.6pt width 23pt, {\em Harnack inequalities
  and {B}\^ocher-type theorems for conformally invariant, fully nonlinear
  degenerate elliptic equations}, Comm. Pure Appl. Math., 67 (2014),
  pp.~1843--1876.

\bibitem{LiNg-arxiv}
\leavevmode\vrule height 2pt depth -1.6pt width 23pt, {\em A fully nonlinear
  version of the {Y}amabe problem on locally conformally flat manifolds with
  umbilic boundary},  ({\noopsort{a}}2009).
\newblock http://arxiv.org/abs/0911.3366v1.

\bibitem{LiNgWang}
{\sc Y.~Y. Li, L.~Nguyen, and B.~Wang}, {\em Comparison principles and
  {L}ipschitz regularity for some nonlinear degenerate elliptic equations},
  (2016).
\newblock https://arxiv.org/abs/1612.09418.

\bibitem{LoewnerNirenberg}
{\sc C.~Loewner and L.~Nirenberg}, {\em Partial differential equations
  invariant under conformal or projective transformations},  (1974),
  pp.~245--272.

\bibitem{Mazzeo:singular-Yamabe}
{\sc R.~Mazzeo}, {\em Regularity for the singular {Y}amabe problem}, Indiana
  Univ. Math. J., 40 (1991), pp.~1277--1299.

\bibitem{Mazzeo-Pacard:construction}
{\sc R.~Mazzeo and F.~Pacard}, {\em A construction of singular solutions for a
  semilinear elliptic equation using asymptotic analysis}, J. Differential
  Geom., 44 (1996), pp.~331--370.

\bibitem{Ou:singularities}
{\sc Q.~Ou}, {\em Singularities and {L}iouville theorems for some special
  conformal {H}essian equations}, Pacific J. Math., 266 (2013), pp.~117--128.

\bibitem{S-Y}
{\sc R.~Schoen and S.-T. Yau}, {\em Conformally flat manifolds, {K}leinian
  groups and scalar curvature}, Invent. Math., 92 (1988), pp.~47--71.

\bibitem{STW07}
{\sc W.-M. Sheng, N.~S. Trudinger, and X.-J. Wang}, {\em The {Y}amabe problem
  for higher order curvatures}, J. Differential Geom., 77 (2007), pp.~515--553.

\bibitem{Sui}
{\sc Z.~Sui}, {\em Complete conformal metrics of negative {R}icci curvature on
  {E}uclidean spaces}, J. Geom. Anal., 27 (2017), pp.~893--907.

\bibitem{TW09}
{\sc N.~S. Trudinger and X.-J. Wang}, {\em On {H}arnack inequalities and
  singularities of admissible metrics in the {Y}amabe problem}, Calc. Var.
  Partial Differential Equations, 35 (2009), pp.~317--338.

\bibitem{Veron81-JDE}
{\sc L.~V\'eron}, {\em Singularit\'es \'eliminables d'\'equations elliptiques
  non lin\'eaires}, J. Differential Equations, 41 (1981), pp.~87--95.

\bibitem{Viac00-Duke}
{\sc J.~A. Viaclovsky}, {\em Conformal geometry, contact geometry, and the
  calculus of variations}, Duke Math. J., 101 (2000), pp.~283--316.

\bibitem{Viac00-AMS}
\leavevmode\vrule height 2pt depth -1.6pt width 23pt, {\em Some fully nonlinear
  equations in conformal geometry}, in Differential equations and mathematical
  physics ({B}irmingham, {AL}, 1999), vol.~16 of AMS/IP Stud. Adv. Math., Amer.
  Math. Soc., Providence, RI, 2000, pp.~425--433.

\end{thebibliography}
\end{document}